\documentclass{article}

\usepackage{amsmath,amssymb,amsthm}

\newcommand{\Q}{\mathbb{Q}}
\newcommand{\Z}{\mathbb{Z}}

\newtheorem{thm}{Theorem}
\newtheorem{lemma}[thm]{Lemma}

\DeclareMathOperator{\rad}{rad}

\begin{document}

\title{On the Diophantine equation 
$X^{2N} + 2^{2\alpha}5^{2\beta}p^{2\gamma} = Z^5$}
\author{Eva G. Goedhart and Helen G. Grundman}

\date{}

\maketitle

\begin{abstract}
We prove that for each odd prime $p$, positive integer
$\alpha$, and non-negative integers $\beta$ and $\gamma$,
the Diophantine equation
$X^{2N} + 2^{2\alpha}5^{2\beta}p^{2\gamma} = Z^5$
has no solution with $X$, $Z$, $N\in\Z^+$, $N > 1$,
and $\gcd(X,Z) = 1$.
\end{abstract}

\section{Introduction}

In 2001, Arif and Abu Muriefah~\cite{ArMu01} (and in 2002, independently,
Le~\cite{Le02}) proved that there is
no integer solution to the equation 
$x^2 + 2^{2m} = y^n$,
with $m \geq 3$, $n \geq 3$, and
$\gcd(x,y) = 1$.  
Since that time, there has been
great interest in studying many variations of this equation.  Of 
particular interest here are those in which
the $2^{2m}$ is replaced by a power of a different prime 
or with the product of a small number of primes raised to powers. 
We consider an equation of the latter form, in which
we also replace the variable exponent in the final term with the 
constant 5 and allow for any even exponent greater than 2
on the first term.  Our equation is actually inspired by 
the work of Bennett~\cite{Be06} in which he considers
the equation $x^{2n} + y^{2n} = z^5$, with $n > 1$.  
We do not require that the middle term be raised to the power $2n$,
only that it be an even square with few prime factors.

\begin{thm}
\label{mainthm}
Let $p$ be an odd prime, $\alpha$ a positive integer,
and $\beta$ and $\gamma$ non-negative integers.
The equation
\begin{equation}
\label{maineq}
X^{2N} + 2^{2\alpha}5^{2\beta}p^{2\gamma} = Z^5
\end{equation}
has no solution with $X$, $Z$, $N\in\Z^+$, $N > 1$,
and $\gcd(X,Z) = 1$.
\end{thm}
Note that the condition $N > 1$ is necessary for the theorem to hold, 
since, for example,
$41^2 + 2^2\cdot 19^2 = 5^5$.  

A number of special cases of Theorem~\ref{mainthm} are already known.
For $N = 2$, Bruin~\cite[Theorem 1.1]{Br03} proved that 
equation~(\ref{maineq}) has no positive integer solutions and for $N = 3$, 
Bennett and Chen~\cite[Theorem 1]{BeCh12} proved likewise.
The theorem is also known to be true for $\beta = 
\gamma = 0$~\cite{ArMu01,Le02},
$\beta \neq 0$ and $\gamma = 0$~\cite{LuTo08},
$p = 3$ and $\beta = 0$~\cite{Lu02}, 
$p = 11$ and $\beta = 0$~\cite{CaDe10}, 
$p = 13$~\cite{GoLuTo08},
$p = 17$~\cite{GoMaTo12}, 
$p = 19$ and $\beta = 0$~\cite{SoUlZh12}, 
$2^{2\alpha}5^{2\beta}p^{2\gamma} 
\leq 100$ (see, for example,~\cite{BuMiSi06}), 
and $N$ divisible by a prime greater than 17 that is congruent
to 1 modulo~4~\cite{Ch10}.

We note further that, 
since equation~(\ref{maineq}) is of the form $X^{2N} + C^2 = Z^5$,
a result of Darmon and Granville~\cite[Theorem 2]{DaGr95} guarantees that,
for any given value of $N$, there are at
most finitely many integer solutions with $\gcd(X,Z) = 1$.

In the following section, we first present and prove a lemma 
important to the proof of
Theorem~\ref{mainthm}.  We then state a simplified version of
a result due to Bennett and Skinner~\cite{BeSk04}, specific
to our needs.  In Section~\ref{proof}, we prove Theorem~\ref{mainthm},
following the ideas and methods found in~\cite{Be06}.

\section{Preliminaries}

We begin with two lemmas.

\begin{lemma}\label{5}
Let $p$, $\alpha$, $\beta$, and $\gamma$ be as in Theorem~\ref{mainthm},
with $p\neq 5$.
Let $u$ and $v\in \Z$ be coprime, with $v$ even, such that
\begin{equation}
\label{lv()}
2^\alpha 5^{\beta}p^\gamma = v(v^4 - 10u^2v^2 + 5u^4).
\end{equation}
Then $v^4 - 10u^2v^2 + 5u^4 \neq 5$.
\end{lemma}

\begin{proof}
Suppose that $v^4 - 10u^2v^2 + 5u^4 = 5$.  Then,
by equation~(\ref{lv()}),
\[v = 2^\alpha 5^{\beta - 1}p^\gamma\] and $\beta \geq 2$.
Combining the two equations, we find that $5(u^2 - v^2)^2 - 5 = 4v^4 
= 2^{4\alpha + 2} 5^{4\beta - 4}p^{4\gamma}$,
and so $(u^2 - v^2 + 1)(u^2 - v^2 - 1) 
= 2^{4\alpha + 2} 5^{4\beta - 5}p^{4\gamma}$.
Since $\gcd(u^2 - v^2 + 1,u^2 - v^2 - 1) = 2$ and
$u^2 - v^2 + 1 \equiv 2\pmod 4$, 
\[u^2 - v^2 + 1 = 2\cdot 5^k p^\ell
\hspace{1pc} \mbox{and} \hspace{1pc}
u^2 - v^2 - 1 = 2^{4\alpha + 1} 5^{k^\prime} p^{\ell^\prime},\]
where $\{k,k^\prime\} = \{0,4\beta - 5\}$ and
$\{\ell,\ell^\prime\} = \{0,4\gamma\}$. 
Subtracting, then dividing by 2, we obtain
\begin{equation}\label{lk}
5^k p^\ell - 2^{4\alpha} 5^{k^\prime} p^{\ell^\prime} = 1.
\end{equation}
Now, $2^4 \equiv 1 \pmod 5$ and, since $p\neq 5$, 
$p^4 \equiv 1 \pmod 5$.  Hence, equation~(\ref{lk}) implies that
$k = 0$.  It follows that $\ell \neq 0$ and so we have
$p^{4\gamma} - 2^{4\alpha} 5^{4\beta - 5} = 1$ with $\gamma \neq 0$.
If $p \neq 3$, then reducing modulo~3 yields a contradiction.
Thus $p = 3$ and 
\begin{equation}\label{l=1}
3^{4\gamma} - 2^{4\alpha} 5^{4\beta - 5} = 1.
\end{equation}
But this provides a positive integer solution to the
equation $X^2 + 2^a\cdot 5^b = Y^N$ with
$\gcd(X,Y) = 1$, $4\mid N$, $a > 0$, and $b \geq 3$, 
contradicting~\cite[Theorem 1.1]{LuTo08}.

Therefore, $v^4 - 10u^2v^2 + 5u^4 \neq 5$.
\end{proof}

Following a ``modular approach" to solving Diophantine equations,
Bennett and Skinner~\cite{BeSk04} 
developed the main tools we use
in proving Theorem~\ref{mainthm}.  We give here a corollary of a
particular case of one of their results, based on the 
presentation given in~\cite[Theorem 15.8.3]{CohenII}.  As usual, for
$a\in \Q$, let $v_p(a)$ denote the $p$-valuation of $a$.

\begin{lemma}[Bennett-Skinner]\label{BeSk}
Let $x^7 + Cy^7 = z^2$ 
with $C$, $x$, $y$, $z\in \Z$, $xy\neq\pm 1$,
$x$, $Cy$, and $z$ nonzero and pairwise relatively prime, 
$z\equiv 1\pmod 4$,
$v_2(Cy^7)\geq 6$,
and for all primes $q$, 
$v_q(C) < 7$.
Then there exists a newform of level
\[
N_7 = \left\{\begin{array}{ll}
2\rad(C),  & \mbox{if\ } v_2(C) = 0,\\
\rad(C)/2,  & \mbox{if\ } v_2(C) = 6,\\
\rad(C), & \mbox{otherwise.}
\end{array}
\right.
\]
\end{lemma}

\section{Proof of Theorem~\ref{mainthm}}
\label{proof}

Let $p$, $\alpha$, $\beta$, and $\gamma$ be as in Theorem~\ref{mainthm}
and suppose that $(N,X,Z) = (n,x,z)$ is a solution to
equation~(\ref{maineq}) with
$n$, $x$, $z\in\Z^+$, $n > 1$, 
and $\gcd(x,z) = 1$. 
Note that, since $\alpha \geq 1$, $x$ and $z$ are both odd.

We assume without loss of generality that $p\neq 5$ and that
$n$ is prime.  As noted in the introduction, 
by~\cite{Br03}, $n \neq 2$, and 
by~\cite{BeCh12}, $n \neq 3$.

Suppose that $n = 5$.
By equation~(\ref{maineq}),
\begin{equation}\label{n=5main}
x^{10} + 2^{2\alpha}5^{2\beta}p^{2\gamma} = z^5, 
\end{equation}
and so
\begin{equation}\label{n=5}
2^{2\alpha}5^{2\beta}p^{2\gamma} = 
\left(z - x^2\right)\left(z^4 + z^3 x^2 + z^2 x^4 + z x^6 + x^8\right).
\end{equation}
Since $x$ and $z$ are odd, $z - x^2$ is even.
Note that, since $x \geq 1$ and $\alpha \geq 1$, 
$x^{10} + 2^{2\alpha}5^{2\beta}p^{2\gamma} \geq 5$, which
implies that $z > 1$, and, therefore,
$z^4 + z^3 x^2 + z^2 x^4 + z x^6 + x^8 \neq 1$ or 5.

If $\beta = 0$, then
$\gcd(z - x^2,z^4 + z^3 x^2 + z^2 x^4 + z x^6 + x^8) = 1$
and so
\begin{equation*}
z - x^2 = 2^{2\alpha}
\mbox{\ and\ }
z^4 + z^3 x^2 + z^2 x^4 + z x^6 + x^8 = p^{2\gamma}. 
\end{equation*}
 
If $\beta \neq 0$, then,
noting that $z - x^2 \equiv z^5 - x^{10} \equiv 0 \pmod 5$, 
we have $5\mid (z - x^2)$.  
So
$z^4 + z^3 x^2 + z^2 x^4 + z x^6 + x^8 \equiv 5 \pmod{25}$ and
$\gcd(z - x^2,z^4 + z^3 x^2 + z^2 x^4 + z x^6 + x^8) = 5$.
Hence, from equation~(\ref{n=5}), 
\begin{equation*}
z - x^2 = 2^{2\alpha}5^{2\beta - 1} \mbox{\ and\ }
z^4 + z^3 x^2 + z^2 x^4 + z x^6 + x^8 = 5p^{2\gamma}.
\end{equation*}

Thus, in either case,
we have $z = x^2 + 2^{2\alpha}5^j$, with $j \geq 0$.
So equation~(\ref{n=5main}) becomes
$2^{2\alpha}5^{2\beta}p^{2\gamma} = (x^2 + 2^{2\alpha}5^j)^5 - x^{10}$.
Expanding and removing a factor of $2^{2\alpha}$, we have
\begin{equation}\label{split}
5^{2\beta} p^{2\gamma} = 5^{j+1} x^8 + 2^{2\alpha+1} 5^{2j+1} x^6 +
 2^{4\alpha+1} 5^{3j+1} x^4 + 2^{6\alpha} 5^{4j+1} x^2
+ 2^{8\alpha} 5^{5j}.
\end{equation}
If $\beta = 0$, then $j = 0$ and 
reducing equation~(\ref{split}) modulo~8
yields $1\equiv 5 \pmod 8$, a contradiction.
If $\beta \neq 0$, then $j = 2\beta - 1$ and
reducing equation~(\ref{split}) modulo~3 yields 
$p^{2\gamma}\equiv 2 \pmod 3$, another contradiction.
Hence, $n\neq 5$.

So $n \geq 7$.

Writing equation~(\ref{maineq}) in the form
$\left(x^n\right)^2 + \left(2^{\alpha}5^{\beta}p^{m}\right)^2 = z^5$,
a classical argument 
(see, for example,~\cite[Section 14.2]{CohenII}) yields
nonzero coprime integers, $u$ and $v$, of opposite parity, such that
\begin{equation}
\label{u()} 
x^n  = u(u^4 - 10u^2v^2 + 5v^4) 
\end{equation}
and
\begin{equation}
\label{v()} 
2^\alpha 5^{\beta}p^\gamma = v(v^4 - 10u^2v^2 + 5u^4).  
\end{equation}
Since $x$ is odd, equation~(\ref{u()}) implies that $u$ is odd.
Since $u$ and $v$ are of opposite parity,
$v$ is even.

Further, since $\gcd(u,v) = 1$,
\begin{equation*}
\gcd(v,v^4 - 10u^2v^2 + 5u^4) = \gcd(v,5).
\end{equation*}

If $5\mid v$, then $5\nmid u$ and 
so $v^4 - 10u^2v^2 + 5u^4 \equiv 5 \pmod {25}$.
Thus, since $\gcd(v,v^4 - 10u^2v^2 + 5u^4) = 5$,
by equation~(\ref{v()}), 
$v^4 - 10u^2v^2 + 5u^4 = 5$ or $\pm 5p^\gamma$.  By
Lemma~\ref{5}, the first is impossible.  Therefore, we have
\begin{equation}
\label{5midv}
v = \pm 2^{\alpha} 5^{\beta - 1}
\hspace{1pc} \mbox{and} \hspace{1pc}
v^4 - 10u^2v^2 + 5u^4 = \pm 5p^\gamma,
\end{equation}
with $\gamma \neq 0$.

If $5\nmid v$, then, by equation~(\ref{v()}), $\beta = 0$. Since
$\gcd(v,v^4 - 10u^2v^2 + 5u^4) = 1$, 
$v^4 - 10u^2v^2 + 5u^4 = \pm 1$ or $\pm p^\gamma$.  But
$v^4 - 10u^2v^2 + 5u^4 \equiv 5 \pmod 8$, since $v$ is even.  
Hence, in this case,
\begin{equation}
\label{5nmidv}
v = \pm 2^{\alpha} 
\hspace{1pc} \mbox{and} \hspace{1pc} 
v^4 - 10u^2v^2 + 5u^4 = \pm p^\gamma.
\end{equation}

Combining equations~(\ref{5midv}) and~(\ref{5nmidv}), we have
\begin{equation}\label{v}
v = \pm 2^{\alpha}5^{k} \hspace{1pc}  \mbox{and}
\hspace{1pc}  v^4 - 10u^2v^2 + 5u^4 =\pm  5^{\beta - k}p^\gamma,
\end{equation}
where $k = \beta - 1$ if $5\mid v$, and $k = 0$ otherwise.

Now, if $5\mid u$, then $5\nmid v$ and we have
$\gcd(u,u^4 - 10u^2v^2 + 5v^4) = 5$.  Since
$u^4 - 10u^2v^2 + 5v^4\equiv 5\pmod {25}$ and $n$ is odd,
by equation~(\ref{u()}),
there exist nonzero coprime integers $A_1$, $B_1\in \Z$ such that
\begin{equation}\label{5u^4}
u = 5^{n-1}A_1^n
\hspace{1pc} \mbox{and} \hspace{1pc}
u^4 - 10u^2v^2 + 5v^4 = 5B_1^n.
\end {equation} 
Thus,
$5B_1^n + 20v^4 = \left(u^2 - 5v^2\right)^2$.
Recalling that $5\nmid v$, we can combine this
with equation~(\ref{5nmidv}),
letting $w_1 = \left(u^2 - 5v^2\right)/5 \in \Z$,
to obtain
\begin{equation}\label{5w_1}
B_1^n + 2^{4\alpha + 2} = 5w_1^2.
\end{equation}

Reducing the second part of~(\ref{5u^4}) modulo 8,
we find that $1\equiv 5B_1^n \pmod 8$, and hence
$B_1$ is odd and not equal to $\pm 1$.  
By~\cite[Theorem 1.2]{BeSk04}, there is no integer
solution to the equation
$X^n + 2^{4\alpha + 2}Y^n = 5Z^2$, satisfying  these conditions.
Thus, we have a contradiction.

On the other hand, if $5\nmid u$, then 
$\gcd(u,u^4 - 10u^2v^2 + 5v^4) = 1$.
This together with equation~(\ref{u()}) and
the fact that $n$ is odd implies that 
there exist nonzero coprime integers $A_2$, $B_2\in\Z$ such that
\begin{equation}\label{u^4}
u = A_2^n
\hspace{1pc} \mbox{and} \hspace{1pc} 
u^4 - 10u^2v^2 + 5v^4 = B_2^n.
\end {equation}
Thus, $B_2^n + 20v^4 = \left(u^2 - 5v^2\right)^2$.  
Combining this with equation~(\ref{v})
and letting $w_2 = u^2 - 5v^2$
yields 
\begin{equation}\label{5w_2} 
B_2^n + 2^{4\alpha+2} 5^{4k+1} = w_2^2.
\end{equation}

By equation~(\ref{u^4}), since $u$ is not divisible by 2 or 5,
neither is $B_2$.  So $\gcd(B_2,w_2) = 1$.  
By~\cite[Theorem 1.5]{BeSk04}, there is no integer
solution to the equation
$X^n + 2^{4\alpha+2} 5^{4k+1}Y^n = Z^2$, satisfying these conditions,
with prime $n \geq 11$.  Hence, we have a contradiction unless
$n = 7$.

For the case $n = 7$ (still assuming that $5\nmid u$), 
we first note that $4\alpha + 2 \geq 6$, $w_2 \equiv 1\pmod 4$,
and, since $\gcd(u,v) = 1$, $3\nmid w_2$.  
Evaluating equation~(\ref{5w_2}) modulo~3, recalling that $n = 7$, yields
$B_2^7 + 2 \equiv 1 \pmod 3$, implying that $B_2\equiv 2 \pmod 3$.
On the other hand, evaluating equation~(\ref{5w_2}) modulo~8 yields
$B_2 \equiv B_2^7 \equiv w_2^2 \equiv 1\pmod 8$. 
Thus $B_2 \neq \pm 1$.

Rewriting equation~(\ref{5w_2}) in the form
\[B_2^7 + 2^{r_1}5^{r_2} (2^{s_1}5^{s_2})^7  = w_2^2,\]
with $r_i$, $s_i\in \Z$ such that 
$0\leq r_i < 7$, for $i \in \{1,2\}$,
we can apply Lemma~\ref{BeSk} (with $C = 2^{r_1}5^{r_2}$). 
Hence, there exists a newform of level $N_7$, where
$N_7 \in \{1,2,5,10\}$.
But, as is well-known 
(see, for example,~\cite[Corollary 15.1.2]{CohenII}),
there are no newforms of any of these levels.  
Therefore $n\neq 7$, yielding another contradiction.

Hence the initial supposition is false, and the theorem is proved.

\vskip 10pt
\noindent
Department of Mathematics \\
Bryn Mawr College \\
Bryn Mawr, PA 19010 \\
egoedhart@brynmawr.edu\\
grundman@brynmawr.edu

\end{document}